\documentclass[12pt,reqno]{article}
\usepackage[usenames]{color}
\usepackage{amssymb}
\usepackage{graphicx}
\usepackage{amscd}
\usepackage{mathdots}
\usepackage[colorlinks=true,
linkcolor=webgreen,
filecolor=webbrown,
citecolor=webgreen]{hyperref}

\definecolor{webgreen}{rgb}{0,.5,0}
\definecolor{webbrown}{rgb}{.6,0,0}

\usepackage{color}
\usepackage{fullpage}
\usepackage{float}

\usepackage{graphics,amsmath,amssymb}
\usepackage{amsthm}
\usepackage{amsfonts}
\usepackage{latexsym}
\usepackage{epsf}

\newcommand{\seqnum}[1]{\href{http://oeis.org/#1}{\underline{#1}}}

\makeatletter
\newcommand*\bigcdot{\mathpalette\bigcdot@{.5}}
\newcommand*\bigcdot@[2]{\mathbin{\vcenter{\hbox{\scalebox{#2}{$\m@th#1\bullet$}}}}}
\makeatother

\begin{document}
\theoremstyle{plain}
\newtheorem{theorem}{Theorem}
\newtheorem{corollary}[theorem]{Corollary}
\newtheorem{lemma}[theorem]{Lemma}
\newtheorem{proposition}[theorem]{Proposition}

\theoremstyle{definition}
\newtheorem{definition}[theorem]{Definition}
\newtheorem{example}[theorem]{Example}
\newtheorem{conjecture}[theorem]{Conjecture}

\theoremstyle{remark}
\newtheorem{remark}[theorem]{Remark}
\newtheorem*{remark-non}{Remark}

\begin{center}
\vskip 1cm{\Large\bf Catalan and Motzkin \\
\vskip .05in
Integral Representations}
\vskip 1cm
Peter McCalla and Asamoah Nkwanta\\
Department of Mathematics\\
Morgan State University\\
Baltimore, MD 21251 \\
USA \\
\href{mailto:peter.mccalla@morgan.edu}{\tt peter.mccalla@morgan.edu}\\
\href{mailto:asamoah.nkwanta@morgan.edu}{\tt asamoah.nkwanta@morgan.edu}\\
\ \\
\end{center}

\vskip .2 in
\begin{abstract}
We present new proofs of eight integral representations of the Catalan numbers. Then, we create analogous integral representations of the Motzkin numbers and obtain new results. Most integral representations of counting sequences found in the literature are proved by using advanced mathematical techniques. All integral representations in this paper are proved by using standard techniques from integral calculus.  Thus, we provide a more simplistic approach of proving integral representations of the Catalan and Motzkin numbers.
\end{abstract}

\section{Introduction}
\label{sec:introduction}

The Catalan numbers \seqnum{A000108} \cite{SP1}  are a sequence of natural numbers that are defined as $C_n = \frac{1}{n+1} {2n\choose n}$  for  $n \in \mathbb{Z}_{\ge 0}=\{0,1,2, \dots\}$. Catalan numbers appear as solutions to several combinatorics problems. See Stanley \cite{STAN, STAN1, STAN2} for a number of combinatorial and analytical interpretations of $C_n$. It is closely related to the Motzkin numbers $M_n$ \seqnum{A001006} \cite{SP1} defined as \cite{BERN}\begin{align}M_n = \sum_{k=0}^{\left\lfloor n/2 \right\rfloor} {n\choose 2k}C_k.  \label{eq1}\end{align}

Finding integral representations of counting numbers is a topic of interest for its intrinsic value. Various integral representations of $C_n$ are given in the literature, derived from methods involving Mellin transforms, Chebyshev polynomials, the Cauchy integral formula, and other advanced techniques that are beyond integral calculus. In this paper, we present several integral representations of $C_n$, then using two new results, prove analogous integral presentations of $M_n$. All representations are proved by using standard techniques from integral calculus. However, the authors do not claim that this paper contains all known integral representations of $C_n$ and $M_n$.

Section 2 surveys the Catalan representations and Section 3 presents the Motzkin representations. What is new in Section 2 are the proofs. Section 3 contains new results: Theorems \ref{thm10} and \ref{thm11} and Corollaries \ref{cor12} and \ref{cor13}.  The only known integral representations of $M_n$ in Section 3 are Corollary \ref{cor12}(e) and (f). They are listed in OEIS \cite{SP1} and there are no references to proofs. The proofs of Corollary \ref{cor12}(e) and (f) are also new. 
\section{Integral representations of $C_n$}

We list eight integral representations of $C_n$. The first representation will be proved directly and subsequent representations will be proved using equivalence. We first prove the following lemma. 

\begin{lemma}\label{lma1} For $r, s \in \mathbb{Z}_{\ge 0}$ and positive real number $a$, \begin{displaymath} \int_0^{a/2} \cos^r\left(\frac{\pi x}{a}\right) \sin^s\left(\frac{\pi x}{a}\right) \, dx = (-1)^r\int_{a/2}^a \cos^r\left(\frac{\pi x}{a}\right) \sin^s\left(\frac{\pi x}{a}\right) \, dx. \end{displaymath}\end{lemma}

\begin{proof} Using the integral identity $\int_0^A f(x) \, dx = \int_0^A f(A-x) \, dx$ for continuous function $f$ and the subtraction trigonometric identities $\sin(M-N)=\sin M\cos N - \cos M\sin N$ and $\cos(M-N)=\cos M\cos N + \sin M\sin N$ gives \begin{displaymath}
\begin{split}
\int_0^{a/2} \cos^r\left(\frac{\pi x}{a}\right) \sin^s \left( \frac{\pi x}{a}\right) \, dx &= \int_0^{a/2} \cos^r\left(\frac{\pi}{2}-\frac{\pi x}{a}\right)\sin^s\left(\frac{\pi}{2}-\frac{\pi x}{a}\right) \, dx \\
                                                                 &= \int_0^{a/2} \sin^r\left(\frac{\pi x}{a}\right) \cos^s\left(\frac{\pi x}{a}\right) \, dx.
\end{split}
\end{displaymath}Let $u=x+\frac{a}{2}$. Then by using the identities $\sin\left(M-\frac{\pi}{2}\right)=-\cos M$ and $\cos\left(M-\frac{\pi}{2}\right)=\sin M$,  \begin{displaymath}
\begin{split}
\int_0^{a/2} \sin^r\left(\frac{\pi x}{a}\right) \cos^s\left(\frac{\pi x}{a}\right) \, dx &= \int_{a/2}^a \sin^r\left(\frac{\pi u}{a} - \frac{\pi}{2}\right)\cos^s\left(\frac{\pi u}{a} - \frac{\pi}{2} \right) \, du \\
                                                                &=(-1)^r \int_{a/2}^a \cos^r\left(\frac{\pi u}{a}\right)  \sin^s\left(\frac{\pi u}{a}\right)  \, du. \qedhere
\end{split}
\end{displaymath}  \end{proof}

\begin{theorem}\label{thm2} \textup{(\cite{NKWA1})} For $n \in \mathbb{Z}_{\ge 0}$, \begin{align} C_n = \frac{4^n}{(n+1)\pi} \int_{-1}^1 \frac{x^{2n}}{\sqrt{1-x^2}} \; dx. \label{eq2} \end{align} \end{theorem}

\begin{proof} We start with the identity \cite{KOSH, WIEN} \begin{displaymath}{2n \choose n} = \frac{2 \cdot 4^n}{\pi} \int_0^{\pi/2} \cos^{2n} x \, dx. \end{displaymath} Using Lemma \ref{lma1} with $r=2n$, $s=0$, and $a=\pi$, \begin{align}{2n\choose n} = \frac{4^n}{\pi} \int_0^{\pi} \cos^{2n}x \, dx. \label{eq3} \end{align} Using the substitution $u=\cos x$, we get $dx=-\frac{1}{\sqrt{1-u^2}} \, du$ and \begin{displaymath}{2n \choose n} = \frac{4^n}{\pi} \int_{-1}^1 \frac{u^{2n}}{\sqrt{1-u^2}} \, du. \end{displaymath} The result follows by multiplying both sides by $\frac{1}{n+1}$. \end{proof}

\begin{theorem}\label{thm3} \textup{(\cite{NKWA1})} For $n \in \mathbb{Z}_{\ge 0}$,\begin{align}C_n = \frac{4^n}{(n+1)\pi} \int_0^1 \frac{x^n}{\sqrt{x-x^2}} \; dx. \label{eq4}\end{align}  \end{theorem}

\begin{proof} The integrand in (\ref{eq2}) is an even function, so \begin{displaymath}\begin{split}\frac{4^n}{(n+1)\pi}\int_{-1}^1 \frac{x^{2n}}{\sqrt{1-x^2}} \; dx &= \frac{4^n}{(n+1)\pi} \int_0^1 \frac{2x^{2n}}{\sqrt{1-x^2}} \; dx  \\
                                                                                               &= \frac{4^n}{(n+1)\pi}  \int_0^1 \frac{x^{2n}}{x\sqrt{1-x^2}} \; 2x \; dx  \\
                                                                                               &= \frac{4^n}{(n+1)\pi} \int_0^1 \frac{(x^2)^n}{\sqrt{(x^2)-(x^2)^2}} \; d(x^2)  \\
                                                                                               &= \frac{4^n}{(n+1)\pi} \int_0^1 \frac{x^n}{\sqrt{x-x^2}} \; dx. \qedhere \end{split} \end{displaymath}\end{proof}

\begin{theorem}\label{thm4} \textup{(\cite{PENS, FQI})} For $n \in \mathbb{Z}_{\ge 0}$, \begin{align}C_n =\frac{1}{2\pi}\int_0^4 x^n \sqrt{\frac{4-x}{x}} \; dx. \label{eq5} \end{align}  \end{theorem}

\begin{proof} From (\ref{eq4}), let $y=4x$. Then, we have \begin{displaymath}\begin{split} \frac{4^n}{(n+1)\pi} \int_0^1 \frac{x^n}{\sqrt{x-x^2}} \; dx &= \frac{4^n}{(n+1)\pi} \int_0^4 \frac{\left(\frac{y}{4}\right)^n}{\sqrt{\frac{y}{4}-\frac{y^2}{16}}} \;  \frac{1}{4} \; dy \\
                                                                                         &= \frac{1}{(n+1)\pi} \int_0^4 \frac{y^n}{\sqrt{\frac{4y-y^2}{16}}} \;  \frac{1}{4} \; dy\\
                                                                                         &= \frac{1}{(n+1)\pi} \int_0^4 \frac{y^n}{\sqrt{4y-y^2}} \; dy \\
                                                                                         &= \frac{1}{(n+1)\pi} \int_0^4 \frac{y^{n+1}}{y^2\sqrt{\frac{4-y}{y}}} \; dy. \end{split}\end{displaymath}Using integration by parts with $u=y^{n+1}$ and $dv= \frac{1}{y^2\sqrt{\frac{4-y}{y}}} \, dy$, \begin{displaymath}\begin{split}\frac{1}{(n+1)\pi} \int_0^4 \frac{y^{n+1}}{y^2\sqrt{\frac{4-y}{y}}} \; dy &= \frac{1}{(n+1)\pi}\left(\frac{n+1}{2}\int_0^4 y^n\sqrt{\frac{4-y}{y}} \; dy\right) \\
                                                                                                            &= \frac{1}{2\pi} \int_0^4  y^n\sqrt{\frac{4-y}{y}} \; dy. \qedhere \end{split}\end{displaymath} \end{proof}

\begin{theorem}\label{thm5} \textup{(\cite{FQI})} For $n \in \mathbb{Z}_{\ge 0}$, \begin{align}C_n = \frac{2^{2n+2}}{\pi}\int_0^{\infty} \frac{x^2}{(1+x^2)^{n+2}} \; dx. \label{eq6}\end{align}\end{theorem}

\begin{proof} Starting from (\ref{eq5}), use the reverse substitution $x=\frac{4}{1+t^2}$. Then, \begin{displaymath}\begin{split}\frac{1}{2\pi}\int_0^4 x^n \sqrt{\frac{4-x}{x}} \; dx
                                                                            &= \frac{1}{2\pi} \int_0^{\infty} \left(\frac{4}{1+t^2}\right)^n \frac{8t^2}{(1+t^2)^2} \; dt \\
                                                                            &= \frac{4^n \cdot 8}{2\pi} \int_0^{\infty} \frac{t^2}{(1+t^2)^{n+2}} \; dt \\
                                                                            &=  \frac{2^{2n+2}}{\pi}\int_0^{\infty} \frac{t^2}{(1+t^2)^{n+2}} \; dt. \qedhere \end{split}\end{displaymath}\end{proof}

\begin{theorem} \label{thm6} \textup{(\cite{YUAN})} For $n \in \mathbb{Z}_{\ge 0}$, \begin{align} C_n = \int_0^1 (2\cos (\pi x))^{2n}(2\sin^2 (\pi x)) \; dx. \label{eq7} \end{align}\end{theorem}

\begin{proof} From (\ref{eq6}), we use a substitution that is useful for integrating rational functions involving sine and cosine --- the Weierstrass substitution $x=\tan\left(\frac{\pi t}{2}\right)$. By constructing a right triangle with angle $\frac{\pi t}{2}$, we get $\cos\left(\frac{\pi t}{2}\right)  = \frac{1}{\sqrt{1+x^2}}$ and $\sin\left(\frac{\pi t}{2}\right)= \frac{x}{\sqrt{1+x^2}}$. Then \begin{displaymath} \frac{2^{2n+2}}{\pi}\int_0^{\infty} \frac{x^2}{(1+x^2)^{n+2}} \; dx = 2^{2n+1} \int_0^1 \cos^{2n}\left(\textstyle \frac{\pi t}{2}\right)\sin^2\left(\textstyle \frac{\pi t}{2}\right) \; dt. \end{displaymath} Let $u=\frac{t}{2}$. Then, \begin{displaymath}\begin{split}2^{2n+1} \int_0^1 \cos^{2n}\left(\textstyle \frac{\pi t}{2}\right)\sin^2\left(\textstyle \frac{\pi t}{2}\right) \; dt &= 2^{2n+2} \int_0^{1/2} \cos^{2n} (\pi u) \sin^2 (\pi u) \, du \\
                                                                                                                                                                   &= 2^{2n+1} \int_0^1 \cos^{2n} (\pi u) \sin^2 (\pi u) \, du \end{split}\end{displaymath} where the last equality is a consequence of Lemma \ref{lma1}.\end{proof}

\begin{theorem} \label{thm7} \textup{(\cite{DANA, NKWA, FQI})} For $n \in \mathbb{Z}_{\ge 0}$, \begin{align} C_n = \frac{2^{2n+5}}{\pi}\int_0^1\frac{x^2(1-x^2)^{2n}}{(1+x^2)^{2n+3}} \; dx. \label{eq8}\end{align}\end{theorem}

\begin{proof} Use Lemma \ref{lma1} to change (\ref{eq7}) back to \begin{displaymath}2^{2n+2}\int_0^{1/2} \cos^{2n} (\pi x) \sin^2 (\pi x) \, dx.\end{displaymath} Let $u=\tan\left(\frac{\pi x}{2}\right)$. Then, $\cos (\pi x) = \frac{1-u^2}{1+u^2}$, $\sin (\pi x) = \frac{2u}{1+u^2}$, $dx= \frac{2}{\pi} \frac{du}{1+u^2}$, and \begin{displaymath}\begin{split}2^{2n+2}\int_0^{1/2} \cos^{2n} (\pi x) \sin^2 (\pi x) \, dx &= \frac{2^{2n+3}}{\pi} \int_0^1 \left(\frac{1-u^2}{1+u^2}\right)^{2n} \left(\frac{2u}{1+u^2}\right)^2  \frac{1}{1+u^2} \, du \\
                                                                                  &= \frac{2^{2n+5}}{\pi} \int_0^1 \frac{u^2(1-u^2)^{2n}}{(1+u^2)^{2n+3}} \, du. \qedhere \end{split}\end{displaymath}\end{proof}

\begin{theorem} \label{thm8}  \textup{(\cite{DANA, FQI})} For $n \in \mathbb{Z}_{\ge 0}$, \begin{align}C_n= \frac{2^{2n+1}}{\pi}\int_{-1}^1 x^{2n}\sqrt{1-x^2} \; dx. \label{eq9} \end{align} \end{theorem}

\begin{proof} From (\ref{eq8}), use the reverse substitution $x=\sqrt{\frac{1-u}{1+u}}$ to get \begin{displaymath}\begin{split}\frac{2^{2n+5}}{\pi}\int_0^1\frac{x^2(1-x^2)^{2n}}{(1+x^2)^{2n+3}} \; dx &= \frac{2^{2n+5}}{\pi} \int_1^0 \frac{\frac{1-u}{1+u}\left(1-\frac{1-u}{1+u}\right)^{2n}}{\left(1+\frac{1-u}{1+u}\right)^{2n+3}} \frac{-1}{(1+u)\sqrt{1-u^2}} \, du \\
&= \frac{2^{2n+5}}{\pi} \int_0^1 \frac{\frac{1-u}{1+u}\left(\frac{2u}{1+u}\right)^{2n}}{\left(\frac{2}{1+u}\right)^{2n+3}(1+u)\sqrt{1-u^2}} \, du \\
&= \frac{2^{2n+2}}{\pi} \int_0^1 \frac{u^{2n}(1-u)(1+u)}{\sqrt{1-u^2}} \, du \\
&= \frac{2^{2n+2}}{\pi} \int_0^1 u^{2n} \sqrt{1-u^2} \, du. \end{split}\end{displaymath} Since the integrand is an even function, the result follows. \end{proof}

\begin{remark-non} This integral representation can be proved using the beta and gamma functions and related identities \cite{FQI}. \end{remark-non} 

\begin{theorem} \label{thm9} \textup{(\cite{FQI})} For $n \in \mathbb{Z}_{\ge 0}$, \begin{align} C_n = \frac{1}{2\pi} \int_{-2}^2 x^{2n}\sqrt{4-x^2} \, dx. \label{eq10} \end{align}\end{theorem}

\begin{proof} From (\ref{eq9}), let $u=2x$. Then \begin{displaymath}\begin{split}  \frac{2^{2n+1}}{\pi}\int_{-1}^1 x^{2n}\sqrt{1-x^2} \; dx &= \frac{2^{2n+1}}{\pi} \int_{-2}^2 \left(\frac{u}{2}\right)^{2n}\sqrt{1-\left(\frac{u}{2}\right)^2} \, \frac{du}{2} \\ 
&= \frac{1}{\pi} \int_{-2}^2 u^{2n}\sqrt{1-\frac{u^2}{4}} \, du \\
&= \frac{1}{2\pi} \int_{-2}^2 u^{2n}\sqrt{4-u^2} \, du. \qedhere \end{split}\end{displaymath}\end{proof}

\section{Integral representations of $M_n$}

The next two theorems will provide a link between the representations (\ref{eq2}) and (\ref{eq4}) - (\ref{eq10}) and Corollaries \ref{cor12} and \ref{cor13}.

\begin{theorem} \label{thm10} Let $n \in \mathbb{Z}_{\ge 0}$ and $a,b \in \mathbb{R} \cup \{\pm \infty\}$. Suppose $C_n$ can be written in the form \begin{displaymath} C_n = \int_a^b \bigl(f(x)\bigr)^{2n}g(x) \, dx\end{displaymath} for integrable functions $f$ and $g$. Then \begin{align}M_n= \frac{1}{2} \int_a^b \bigl((1+f(x))^n+(1-f(x))^n\bigr) g(x) \, dx.  \label{eq11}\end{align}\end{theorem}

\begin{proof}
Using (\ref{eq1}) and the Binomial Theorem, we have \begin{displaymath}\begin{split}
M_n &= \int_a^b \sum_{k=0}^{\left\lfloor n/2 \right\rfloor} {n\choose 2k} \bigl(f(x)\bigr)^{2k}g(x) \, dx \\
        &= \frac{1}{2}\int_a^b \left(\sum_{k=0}^{n} {n\choose k}(f(x))^k + \sum_{k=0}^n {n\choose k}(-1)^k(f(x))^k\right)g(x) \, dx \\
        &= \frac{1}{2} \int_a^b \bigl((1+f(x))^n+(1-f(x))^n\bigr) g(x) \, dx. \qedhere
\end{split}\end{displaymath} 
\end{proof}

\begin{theorem}\label{thm11}Let $n \in \mathbb{Z}_{\ge 0}$ and $a,b \in \mathbb{R} \cup \{\pm \infty\}$. Suppose $C_n$ can be written in the form \begin{displaymath} C_n = \frac{1}{n+1}\int_a^b \bigl(f(x)\bigr)^{2n}g(x) \, dx\end{displaymath} for integrable functions $f$ and $g$. Then \begin{align}M_n= \int_a^b \frac{\varphi_{n+2}(x)-\varphi_{n+1}(x)}{(f(x))^2}g(x) \, dx,  \label{eq12}\end{align}
where \begin{displaymath}\varphi_n(x) = \frac{1}{n}\bigl((1+f(x))^n+(1-f(x))^n-2\bigr)\end{displaymath}\end{theorem}
 \begin{proof} We use the same techniques as the proof of Theorem \ref{thm10}; although, we need to initially convert $M_n$ to a double integral. From (\ref{eq1}), we have
 \begin{displaymath}\begin{split}
M_n &= \int_a^b \sum_{k=0}^{\left\lfloor n/2 \right\rfloor} {n\choose 2k} \frac{1}{k+1} \bigl(f(x)\bigr)^{2k}g(x) \, dx \\
        &= \int_a^b \int_0^1 \sum_{k=0}^{\left\lfloor n/2 \right\rfloor} {n\choose 2k} (f(x))^{2k}y^k g(x) \, dy \, dx \end{split}\end{displaymath}
\begin{displaymath}\begin{split}
        &= \int_a^b \int_0^1 \sum_{k=0}^{\left\lfloor n/2 \right\rfloor}  {n\choose 2k}  (f(x)\sqrt{y})^{2k}g(x) \, dy \, dx \\
        &= \frac{1}{2} \int_a^b g(x) \int_0^1 \left(\sum_{k=0}^n {n\choose k} (f(x)\sqrt{y})^k + \sum_{k=0}^n {n\choose k}(-1)^k (f(x)\sqrt{y})^k\right) \, dy \, dx \\
        &= \frac{1}{2} \int_a^b g(x) \int_0^1 \bigl((1+f(x)\sqrt{y})^n+(1-f(x)\sqrt{y})^n\bigr) \, dy \, dx \\
        &= \int_a^b g(x) \frac{1}{(f(x))^2} \left(\int_1^{1+f(x)}(u^{n+1}-u^n) \, du + \int_1^{1-f(x)} (v^{n+1}-v^n) \, dv\right) \, dx,
\end{split}\end{displaymath} where $u=1+f(x)\sqrt{y}$ and $v=1-f(x)\sqrt{y}$. The result follows by evaluating the inner integrals and simplifying. \end{proof}

\begin{corollary} \label{cor12} For $n \in \mathbb{Z}_{\ge 0}$, \\

(a) $\displaystyle M_n = \frac{1}{4\pi} \int_0^4 \left(\bigl(1+\sqrt{x}\bigr)^n+\bigl(1-\sqrt{x}\bigr)^n\right) \sqrt{\frac{4-x}{x}} \, dx$, \\

(b) $\displaystyle M_n = \frac{2}{\pi} \int_0^{\infty} \left(\left(1+\frac{2}{\sqrt{1+x^2}}\right)^n+\left(1-\frac{2}{\sqrt{1+x^2}}\right)^n\right) \frac{x^2}{(1+x^2)^2} \, dx$, \\

(c) $\displaystyle M_n = \int_0^1 \left(\bigl(1+2\cos \pi x\bigr)^n+\bigl(1-2\cos \pi x\bigr)^n\right) \sin^2 \pi x \, dx$, \\

(d) $\displaystyle M_n = \frac{16}{\pi}\int_0^1 \frac{x^2\bigl((3-x^2)^n+(3x^2-1)^n\bigr)}{(1+x^2)^{n+3}} \, dx$, \\

(e) $\displaystyle M_n = \frac{2}{\pi}\int_{-1}^1 (1+2x)^n\sqrt{1-x^2} \, dx$, \\

(f) $\displaystyle M_n = \frac{1}{2\pi} \int_{-2}^2 (1+x)^n \sqrt{4-x^2} \, dx.$\end{corollary}

\begin{remark-non} The integral representations in Corollary \ref{cor12}(e) and (f) are presented by Peter Lushny and Paul Barry, respectively, in OEIS \cite{SP1}. \end{remark-non} 

\begin{proof} (a) Use (\ref{thm11}) on (\ref{eq5}) with $f(x)=\sqrt{x}$ and $g(x)=\frac{1}{2\pi}\sqrt{\frac{4-x}{x}}$. \\

(b) Use (\ref{thm11}) on (\ref{eq6}) with $f(x)=\frac{2}{\sqrt{1+x^2}}$ and $g(x)=\frac{4x^2}{\pi(1+x^2)^2}$. \\

(c) Use (\ref{thm11}) on (\ref{eq7}) with $f(x)=2\cos \pi x$ and $g(x)=2\sin^2\pi x$. \\

(d) Using (\ref{thm11}) on (\ref{eq8}) with $f(x)=\frac{2(1-x^2)}{1+x^2}$ and $g(x)=\frac{32x^2}{\pi(1+x^2)^3}$, we get \begin{displaymath}\begin{split}
M_n &= \frac{16}{\pi}\int_0^1 \left(\left(1+\frac{2-2x^2}{1+x^2}\right)^n + \left(1-\frac{2-2x^2}{1+x^2}\right)^n\right) \frac{x^2}{(1+x^2)^3} \, dx \\
        &= \frac{16}{\pi} \int_0^1 \left(\left(\frac{3-x^2}{1+x^2}\right)^n+\left(\frac{3x^2-1}{1+x^2}\right)^n\right) \frac{x^2}{(1+x^2)^3} \, dx \\
        &=  M_n = \frac{16}{\pi}\int_0^1 \frac{x^2\bigl((3-x^2)^n+(3x^2-1)^n\bigr)}{(1+x^2)^{n+3}} \, dx.
\end{split}\end{displaymath}

(e) Using (\ref{thm11}) on (\ref{eq9}) with $f(x)=2x$ and $g(x)=\frac{2}{\pi}\sqrt{1-x^2}$, we get \begin{displaymath}\begin{split}
M_n &= \frac{1}{\pi}\int_{-1}^1\bigl((1+2x)^n+(1-2x)^n\bigr) \sqrt{1-x^2} \, dx \\
        &= \frac{1}{\pi}\left(\int_{-1}^1(1+2x)^n\sqrt{1-x^2} \, dx + \int_{-1}^1(1-2x)^n\sqrt{1-x^2} \, dx\right).
\end{split}\end{displaymath} Perform the substitution $u=-x$ on the second integral to get the desired result. \\

(f) Using (\ref{thm11}) on (\ref{eq10}) with $f(x)=x$ and $g(x)=\frac{1}{2\pi}\sqrt{4-x^2}$, we get \begin{displaymath}\begin{split}
M_n &= \frac{1}{4\pi}\int_{-2}^2\bigl((1+x)^n+(1-x)^n\bigr) \sqrt{4-x^2} \, dx \\
        &= \frac{1}{4\pi}\left(\int_{-2}^2(1+x)^n\sqrt{4-x^2} \, dx + \int_{-2}^2(1-x)^n\sqrt{4-x^2} \, dx\right).
\end{split}\end{displaymath}Perform the substitution $u=-x$ on the second integral to get the desired result. \\

\emph{Alternative proof of Corollary \ref{cor12}(f):} From the representation in (e), perform the substitution $u=2x$ to get \begin{displaymath}\begin{split}
\frac{1}{\pi}\int_{-2}^2(1+u)^n\sqrt{1-\frac{u^2}{4}} \, du &= \frac{1}{\pi} \int_{-2}^2 (1+u)^n \sqrt{\frac{4-u^2}{4}} \, du \\
                                                                                       &=  \frac{1}{2\pi} \int_{-2}^2 (1+u)^n \sqrt{4-u^2} \, du. \qedhere
\end{split}\end{displaymath}\end{proof}

\begin{corollary}\label{cor13} For $n \in \mathbb{Z}_{\ge 0}$, \\

(a) $\displaystyle M_n = \frac{1}{4\pi}\int_0^1 \frac{\phi_{n+2}(x)-\phi_{n+1}(x)}{x\sqrt{x-x^2}} \, dx$, \\
 
(b) $\displaystyle M_n = \frac{1}{2\pi}\int_{-1}^1 \frac{\psi_{n+2}(x)-\psi_{n+1}(x)}{x^2\sqrt{1-x^2}} \, dx$,  \\

where \begin{displaymath}\phi_n(x) = \frac{1}{n}\bigl((1+2\sqrt{x})^n+(1-2\sqrt{x})^n-2\bigr)\end{displaymath} and \begin{displaymath}\psi_n(x) = \frac{1}{n}\bigl((1+2x)^n-1\bigr).\end{displaymath}\end{corollary}

\begin{proof} (a) Use (\ref{eq12}) on (\ref{eq4}) with $f(x)=2\sqrt{x}$ and $g(x)=\frac{1}{\pi\sqrt{x-x^2}}$. \\

(b) Use (\ref{eq12}) on (\ref{eq2}) with $f(x)=2x$ and $g(x)=\frac{1}{\pi\sqrt{1-x^2}}$. The result follows by splitting the integral, then using the substitution $u=-x$ on the integrals with the expressions $(1-2x)^{n+2}$ and $(1-2x)^{n+1}$. 
\end{proof} 

\section{Conclusion} 

We give other known Catalan integral representations \cite{NKWA1}: 
\begin{equation*}
C_n =\frac{2^{2n+1}}{\left( 2n+1\right) \pi}\int_{-1}^{1}\frac{x^{2n+2}}{%
\sqrt{1-x^{2}}}\, dx \end{equation*} and \begin{equation*}
 C_n=\frac{4^{n}}{n \pi}\int_{0}^{1}\frac{2x^{n+1}-x^{n}}{\sqrt{x-x^{2}}} \, dx.
\end{equation*}%
Also, from (\ref{eq3}), \begin{equation*}
C_n = \frac{ 4^n}{(n+1)\pi}\int_0^{\pi} \cos^{2n}x \, dx. \end{equation*}See Qi and Guo \cite{FQI} for integral representations of $C_n$ and $\frac{1}{C_n}$. 

We have presented several integral representations of $C_n$ and $M_n$. Finding integral representations of other counting numbers such as the Riordan numbers \seqnum{A005043}, binomial coefficients of the Pascal triangle \seqnum{A109906}, Fine numbers \seqnum{A000957}, and RNA numbers \seqnum{A110320} \cite{SP1} would be of interest. See Qi, Shi, and Guo \cite{FQI1} for interesting integral representations of the little \seqnum{A001003} and large \seqnum{A006318} Schr\"{o}der numbers \cite{SP1}; Dilcher \cite{DILC}, for the even Fibonacci numbers \seqnum{A014445} \cite{SP1}; and Glasser and Zhou \cite{GLAS}, for the Fibonacci numbers \seqnum{A000045} \cite{SP1}.

Given the many combinatorial and analytic interpretations of $C_n$ and $M_n$, an important and interesting problem worth pursuing is to find related analytical interpretations of the integral representations.

\section{Acknowledgment}

The authors would like to thank Rodney Kerby, Leon Woodson, and Guoping Zhang for useful comments on various drafts of the manuscript.

\bigskip
\hrule
\bigskip

\noindent {\it Keywords}: Catalan numbers,  Motzkin numbers, integral representation, Weierstrass substitution.

\bigskip
\hrule
\bigskip

\noindent (Concerned with sequences 
\seqnum{A000045},
\seqnum{A000108}, 
\seqnum{A000957},
\seqnum{A001003},
\seqnum{A001006}, 
\seqnum{A005043},
\seqnum{A006318},
\seqnum{A014445},
\seqnum{A109906},
and \seqnum{A110320}.)

\bigskip
\hrule

\end{document}